\documentclass{amsart}
\usepackage{amscd,amsfonts,amssymb,amsmath}
\usepackage{hyperref}
\usepackage[margin=3.4cm]{geometry}
\usepackage{graphicx}

\newtheorem{theorem}{Theorem}[section]
\newtheorem{cor}[theorem]{Corollary}
\newtheorem{lemma}[theorem]{Lemma}
\newtheorem{prop}[theorem]{Proposition}

\theoremstyle{definition}
\newtheorem{definition}[theorem]{Definition}

\theoremstyle{plain}
\numberwithin{equation}{section}
\newtheorem*{ack}{Acknowledgement}

\newtheorem{corollary}[theorem]{Corollary}

\def \Z{\mathbb Z}

\newcommand{\secref}[1]{Section~\ref{#1}}
\newcommand{\thmref}[1]{Theorem~\ref{#1}}
\newcommand{\lemref}[1]{Lemma~\ref{#1}}

\newcommand{\propref}[1]{Proposition~\ref{#1}}
\newcommand{\corref}[1]{Corollary~\ref{#1}}

\numberwithin{equation}{subsection}

\begin{document}
\title[Twin Groups and Grothendieck's Cartographical groups]{Commutator Subgroups of Twin Groups and Grothendieck's Cartographical groups}
\author[Soumya Dey]{Soumya Dey}
\author[Krishnendu Gongopadhyay]{Krishnendu Gongopadhyay}
\address{Indian Institute of Science Education and Research (IISER) Mohali, Sector 81,  SAS Nagar, P. O. 
Manauli, Punjab 140306, India.}
\email{soumya.sxccal@gmail.com} 
\address{Indian Institute of Science Education and Research (IISER) Mohali, Sector 81,  SAS Nagar, P. O. 
Manauli, Punjab 140306, India.}
\email{krishnendu@iisermohali.ac.in, krishnendug@gmail.com} 
\subjclass[2010]{Primary 20F36; Secondary 20F12, 20F05, 11G32, 05E15    }
\keywords{doodle, twin group, cartographical group, combinatorial maps, commutator subgroup}

\date{\today}

\begin{abstract}
Let $TW_n$ be the twin group on $n$ arcs, $n \geq 2$. The group $TW_{m+2}$ is isomorphic to Grothendieck's $m$-dimensional cartographical group $\mathcal C_m$, $m \geq 1$. In this paper we give a finite presentation for the commutator subgroup $TW_{m+2}'$, and prove that $TW_{m+2}'$ has rank $2m-1$.  We derive that  $TW_{m+2}'$ is free if and only if $m \leq 3$. From this it follows that $TW_{m+2}$ is word-hyperbolic and does not contain a surface group if and only if $m \leq 3$.  It also follows that the automorphism group of $TW_{m+2}$ is finitely presented for $m \leq 3$.  
\end{abstract}
\maketitle

\section{Introduction}
Let $n \geq 2$. The  \emph{twin group on $n$ arcs},  denoted by $TW_n$, is generated by a set of $(n-1)$ generators: $\{ \tau_i \ | \ i=1, 2, \ldots, n-1\}$ satisfying the following set of defining relations:
\begin{equation}\tau_i^2=1, \ \hbox{ for all } i, \end{equation}
\begin{equation}\tau_i \tau_j=\tau_j \tau_i, \ \hbox{ if } |i-j|>1.\end{equation}
 
\medskip The role of this group in the theory of `doodles' on a closed oriented surface is similar to the role of Artin's braid groups in the theory of knots and links. In \cite{khov}, Khovanov investigated the doodle groups, and introduced the twin group of $n$ arcs. Khovanov proved that the closure of a twin is a doodle on the ($2$ dimensional) sphere, see \cite{khov} for details. 

\medskip The above group presentation is also of importance in the Grothendieck's theory of `dessins d'enfant'. For $m \geq 1$, the group $TW_{m+2}$ is isomorphic to Grothendieck's $m$-dimensional cartographical group $\mathcal C_m$. Voevodsky used  this group in \cite{vv} as a generalization of the $2$-dimensional cartographical group. It is a standard fact in this theory that  the conjugacy classes of  the $2$-dimensional cartographical group $\mathcal C_2$  can be identified with combinatorial maps on connected surfaces, not necessarily orientable or without boundary, see \cite{js} for more details.  In \cite{vince1, vince2}, Vince looked at the group $\mathcal C_m$ as `combinatorial maps' and investigated certain topological and combinatorial structures associated to this group.

 \medskip  The commutator subgroup or derived subgroup $G'$ of a group $G$ is generated by the elements of the form $x^{-1} y^{-1} x y$. This subgroup is one measure to know how far $G$ is from being abelian. This is the smallest normal subgroup that abelianize $G$, i.e. the quotient $G/G'$ is abelian. The quotient $G/G'$  also gives the first homology group of $G$. 

\medskip The commutator subgroup $B_n'$ of Artin's braid group on $n$ strands $B_n$ is well-studied.  Gorin and Lin \cite{gl} obtained a finite presentation for  $B_n'$. Several authors have investigated commutators of  larger class of spherical Artin groups, e.g.  \cite{zinde}, \cite{mr},  \cite{orevkov}. 

\medskip In this paper, we ask for the commutator subgroup of the group $TW_n$. Note that $TW_2$ is the cyclic group of order two, and hence the commutator subgroup $TW_2'$ is trivial. However, for $n \geq 3$, the structure of the commutator subgroup $TW_n'$ is non-trivial. It is easy to see that $TW_n'$ is a finite index subgroup of the finitely presented group $TW_n$, hence it is clear that $TW_n'$ is finitely presented. In general, it is a difficult problem to obtain a finite presentation for a finitely presented group, and sometimes it is algorithmically impossible as well, see \cite{bw}. So, knowing that $TW_n'$ is finitely presented is not enough to have a clear understanding about the structure of the group. In this paper, we obtain an explicit finite presentation for $TW_n'$. Since $TW_{m+2}$ is isomorphic to $\mathcal C_{m}$ for all $m \geq 1$, this also gives finite presentation for the group $\mathcal C_m'$.  
\begin{theorem}\label{mainth}
 For $m \ge 1$, $TW_{m+2}'$ has the following presentation:\\
 
 Generators: $ \ \ \ \ \beta_{p}(j), \ \ \ \ \  0 \le p < j \le m. $\\
 
 Defining relations: \ \ \ For all $~ l \ge 3,~~ 1 \le k \le j,~~ j+2 \le t \le m,$
 \begin{equation*}
 \beta_{j-k}(j) ~ \beta_{t-(j+l)}(t) =  \beta_{t-(j+l)}(t) ~ \beta_{j-k}(j), 
 \end{equation*}
 \begin{equation*}
 \beta_{t-k}(t) = \beta_{j-k}(j)^{-1} ~ \beta_{t-(j+1)}(t) ~ \beta_{j-k}(j).
 \end{equation*}
\end{theorem}

\medskip Even if a group is finitely generated, it is a non-trivial problem to compute its rank, that is the smallest cardinality of a generating set for the group. In \cite{pv}, Panov and Ver\"evkin constructed classifying spaces for the commutator subgroups of  right-angled Coxeter groups and have given a general formula for the rank of such groups, see \cite[Theorem 4.5]{pv}.  However, the number of minimal generators given in \cite{pv} is in general form and involves the rank of the zeroth homology groups of certain subcomplexes of  the underlying classifying space. As an immediate application of \thmref{mainth}, we obtain the rank of $TW_n'$ in terms of the `arcs' of the twin group, or the `dimension' of the cartographical group, and thus it is more explicit in our context. We have the following.
\begin{theorem}\label{thmrank}
For $m \geq 1$, the group $TW_{m+2}'$  has rank $2m-1$.	
\end{theorem} 

The following is a consequence of the above two theorems.

\begin{corollary}\label{cor1}
	For $m \ge 1$, the quotient group $~ TW_{m+2}'/TW_{m+2}''$, is isomorphic to the free abelian group of rank $~ 2m-1$, i.e. the group $~ \bigoplus_{i=1}^{2m-1} \Z.$ In particular,  $TW_{m+2}'$ is not perfect for any $m \ge 1$.
\end{corollary}

We further characterize freeness of $TW_n'$ in the following corollary. 

\begin{cor}\label{corfree}
	$TW_{m+2}'$ is a free group if and only if $m \le 3$. The group $TW_3'$ is infinite cyclic. The groups $TW_4'$ and $TW_5'$ are free groups of rank $3$ and $5$ respectively. 
\end{cor}

As applications to the above results, we derive geometric properties of the ambient group $TW_{m+2}$. It is clear from the presentation in \thmref{mainth} that for $m \geq 4$, $TW_{m+2}'$ contains free abelian subgroups of rank $\geq 2$. By \cite[Theorem B]{mou}, this shows that $TW_{m+2}$ is not word-hyperbolic for $m \ge 4$. Whereas from \corref{corfree} we observe that $TW_{m+2}$ is virtually free for $m \le 3$; so it is clear that $TW_{m+2}$ is word-hyperbolic for $m \le 3$. Hence we have the following characterization for word-hyperbolicity of $TW_{m+2}$. 
\begin{cor}\label{wh}
The group $TW_{m+2}$ is word-hyperbolic if and only if $m \leq 3$. 
\end{cor}

Gordon, Long and Reid proved in \cite{glr} that a coxeter group $G$ is virtually free if and only if $G$ does not contain a surface group. Since $TW_{m+2}$ is finitely generated, by \corref{wh}, it can not be virtually free for $m \geq 4$. Hence we have the following.
\begin{cor}
The group $TW_{m+2}$ does not contain a surface group if and only if $m \leq 3$. 
\end{cor}

According to \cite[Theorem B]{kala}, and also \cite[Theorem 1]{krist}, any finite extension of a free group of finite rank has a finitely presented automorphism group. Noting that $TW_n/TW_n'$ is a finite group and using \corref{corfree} we have an immediate corollary as follows. 
\begin{corollary}\label{cor5}
The automorphism group of $TW_{m+2}$ is  finitely presented for $m \leq 3$.
\end{corollary}

\medskip We have proved \thmref{mainth} by systematic use of the Reidemeister-Schreier algorithm. This method is a well-known technique to obtain presentations for subgroups, for details see \cite{mks}.  This algorithm has been used to obtain presentations for certain classes of generalized braid groups and Artin groups in  \cite{bgn1}, \cite{dg1}, \cite{lo},  \cite{man}. We obtain a presentation for $TW_n'$, $n \geq 3$,  using this approach and then remove some of the generators using Tietze transformations. This gives the finite presentation for $TW_n'$. We further reduce the number of generators in this presentation to obtain the rank. 

\medskip  Now we briefly describe the structure of the paper. In \secref{gen}, we compute a generating set for $TW_n'$, $n \geq 3$, using the Reidemeister-Schreier method. In \secref{dr}, a set of defining relations for $TW_n'$ involving these generators is obtained. We then apply Tietze transformations to prove \thmref{mainth} in \secref{simp}. Following this theorem, in \secref{simp}, we also prove \thmref{thmrank}, \corref{cor1} and \corref{corfree}. 

\section{A Generating Set for $TW_n'$}\label{gen}

For $n \geq 3$, define the following map: 
\begin{equation*}\phi : TW_n \longrightarrow \ \underbrace{\Z_2 \oplus \Z_2 \oplus \dots \oplus \Z_2}_\text{(n -- 1) copies} = \bigoplus_{i=1}^{n-1} \Z_2 \end{equation*}
where, for $ i = 1, \ldots , n-1  $, \ $\phi$ maps $\tau_i$ to the generator of the \ $i$ th copy of $\Z_2$ in the product \  $ \bigoplus_{i=1}^{n-1} \Z_2 $ .

Here, Image($\phi$) is isomorphic to the abelianization of $TW_n$, denoted as $TW_n^{ab}$. To prove this, we abelianize the above presentation for $TW_n$ by inserting the relations $ ~ \tau_i \tau_j=\tau_j \tau_i ~ $ (for all $i,j$) in the presentation. The resulting presentation is the following:\\
$$\langle \tau_1, \dots , \tau_{n-1} ~ | ~ \tau_i \tau_j=\tau_j \tau_i, ~ \tau_i^2=1, ~ i,j \in \{ 1,2, \dots n-1 \} \rangle.$$

Clearly, the above is a presentation for $\bigoplus_{i=1}^{n-1} \Z_2$. Thus,  $TW_n^{ab}$ is isomorphic to $\bigoplus_{i=1}^{n-1} \Z_2$. But as $\phi$ is onto, Image($\phi$) = $\bigoplus_{i=1}^{n-1} \Z_2$, i.e. Image($\phi$) is isomorphic to $TW_n^{ab}$. Hence, we get the following short exact sequence:
\begin{equation*}\label{se1}1 \xrightarrow {} TW_n' \hookrightarrow{} TW_n \xrightarrow{\phi} \ \bigoplus_{i=1}^{n-1} \Z_2 \ \xrightarrow{} 1.\end{equation*}  

\begin{lemma}\label{lemma0}
 For $n \ge 3$, $TW_n'$ is generated by the conjugates of $ \ \tau_j \tau_{j+1} \tau_j \tau_{j+1} $ and $ \ \tau_{j+1} \tau_j \tau_{j+1} \tau_j $ by the elements $ \ \tau_{i_1} \tau_{i_2} \dots \tau_{i_s} \ $ for all $ \ j \in \{ 1, 2, \dots , n-2 \} $ and $~ 1 \le i_1 < i_2 < \dots < i_s < j$.
\end{lemma}
\begin{proof} Consider a Schreier set of coset representatives:
$$\Lambda=\{ \tau_1^{\epsilon_1} \tau_2^{\epsilon_2} \dots \tau_{n-1}^{\epsilon_{n-1}} \ | \ \epsilon_i \in \{0, 1\}, \ i=1,2, \dots ,n-1 \}.$$

For $a \in TW_n$, we denote by $\overline{a}$ the unique element in $\Lambda$ which belongs to the coset corresponding to $\phi(a)$ in the quotient group $TW_n/TW_n'$.\\

By \cite[Theorem 2.7]{mks}, the group $TW_n'$ is generated by the set
$$\{S_{\lambda, a}=(\lambda a) (\overline{\lambda a})^{-1} \ | \ \lambda \in \Lambda, \ a \in \{ \ \tau_i \ | \ i=1, 2, \ldots, n-1\} \ \}.$$

Hence, $TW_n'$ is generated by the elements: $S_{\tau_{i_1} \tau_{i_2} \dots \tau_{i_k}, \tau_j }$ for $1 \le i_1 < i_2 < \dots < i_k \le n-1$ and for $1 \le j \le n-1$. We calculate these elements below.

\subsection*{Case 1: \ $i_k \le j$ :}

In this case, $S_{\tau_{i_1} \tau_{i_2} \dots \tau_{i_k}, \tau_j } \ = \ \tau_{i_1} \tau_{i_2} \dots \tau_{i_k} \tau_j \overline{ ( \tau_{i_1} \tau_{i_2} \dots \tau_{i_k} \tau_j )}^{-1} $ \\ 
$ \hbox{\ \ \ \ \ \ \ \ \ \ \ \ \ \ \ \ \ \ \ \ \ \ \ \ \ \ \ \ \ \ \ \ \ \ \ \ \ \ \ \ \ \ \ \ \ \ \ \ \ \ \ \ \ \ \ \ \ \ \ \ \ \ \ \ } = \tau_{i_1} \tau_{i_2} \dots \tau_{i_k} \tau_j \ ( \tau_{i_1} \tau_{i_2} \dots \tau_{i_k} \tau_j )^{-1}=1$.\\

Hence we don't get any nontrivial generator from this case.

\subsection*{Case 2: \ $i_k > j$ :} We divide this case into following 3 subcases.

\subsubsection*{Subcase 2A: \ $i_k > j$ and $ (j+1) \in \{ i_1, i_2, \dots , i_k \} $ but $ j \notin \{ i_1, i_2, \dots , i_k \} $: \\\\}

Suppose $j+1=i_{s+1}$. Then we have:

$S_{\tau_{i_1} \tau_{i_2} \dots \tau_{i_k}, \tau_j } \ = \ \tau_{i_1} \tau_{i_2} \dots \tau_{i_s} \tau_{j+1} \tau_{i_{s+2}} \dots \tau_{i_k} \tau_j \ \overline{ ( \tau_{i_1} \tau_{i_2} \dots \tau_{i_s} \tau_{j+1} \tau_{i_{s+2}} \dots \tau_{i_k} \tau_j )}^{-1} $ \\
$= \tau_{i_1} \tau_{i_2} \dots \tau_{i_s} \tau_{j+1} \tau_{i_{s+2}} \dots \tau_{i_k} \tau_j \ ( \tau_{i_1} \tau_{i_2} \dots \tau_{i_s} \tau_j \tau_{j+1} \tau_{i_{s+2}} \dots \tau_{i_k} )^{-1}$\\
$= \tau_{i_1} \tau_{i_2} \dots \tau_{i_s} \tau_{j+1} \tau_j \tau_{i_{s+2}} \dots \tau_{i_k} \ ( \tau_{i_1} \tau_{i_2} \dots \tau_{i_s} \tau_j \tau_{j+1} \tau_{i_{s+2}} \dots \tau_{i_k} )^{-1}$\\
$= \tau_{i_1} \tau_{i_2} \dots \tau_{i_s} \tau_{j+1} \tau_j \tau_{i_{s+2}} \dots \tau_{i_k} \tau_{i_k} \dots \tau_{i_{s+2}} \tau_{j+1} \tau_j \tau_{i_s} \dots \tau_{i_2} \tau_{i_1} $\\
$= \tau_{i_1} \tau_{i_2} \dots \tau_{i_s} \tau_{j+1} \tau_j \tau_{j+1} \tau_j \tau_{i_s} \dots \tau_{i_2} \tau_{i_1} $.\\

(Here we assume $i_1 < (j+1) < i_k$. The cases $(j+1)=i_1, i_k$ are similar and give same form of elements.)\\

So, we get some of the generators for $TW_n'$ as follows:\\
$ \{ \tau_{i_1} \tau_{i_2} \dots \tau_{i_s} (\tau_{j+1} \tau_j \tau_{j+1} \tau_j) \tau_{i_s} \dots \tau_{i_2} \tau_{i_1} ~ | ~ j \in \{ 1, 2, \dots n-2 \} $ and $i_1 < i_2 < \dots < i_s < j$ where $i_1, i_2, \dots ,i_s,j$ are consecutive integers $ \} $. 

\subsubsection*{Subcase 2B: \ $i_k > j$ and $ j, (j+1) \in \{ i_1, i_2, \dots , i_k \} $: \\\\ }

Suppose $j=i_s, \ j+1=i_{s+1}$. Then we have:

$S_{\tau_{i_1} \tau_{i_2} \dots \tau_{i_k}, \tau_j } \ = \ \tau_{i_1} \tau_{i_2} \dots \tau_{i_{s-1}} \tau_j \tau_{j+1} \tau_{i_{s+2}} \dots \tau_{i_k} \tau_j \ \overline{ ( \tau_{i_1} \tau_{i_2} \dots \tau_{i_{s-1}} \tau_j \tau_{j+1} \tau_{i_{s+2}} \dots \tau_{i_k} \tau_j )}^{-1} $ \\
$= \tau_{i_1} \tau_{i_2} \dots \tau_{i_{s-1}} \tau_j \tau_{j+1} \tau_{i_{s+2}} \dots \tau_{i_k} \tau_j \ ( \tau_{i_1} \tau_{i_2} \dots \tau_{i_{s-1}} \tau_{j+1} \tau_{i_{s+2}} \dots \tau_{i_k} )^{-1}$\\
$= \tau_{i_1} \tau_{i_2} \dots \tau_{i_{s-1}} \tau_j \tau_{j+1} \tau_j \tau_{i_{s+2}} \dots \tau_{i_k} \ ( \tau_{i_1} \tau_{i_2} \dots \tau_{i_{s-1}} \tau_{j+1} \tau_{i_{s+2}} \dots \tau_{i_k} )^{-1}$\\
$= \tau_{i_1} \tau_{i_2} \dots \tau_{i_{s-1}} \tau_j \tau_{j+1} \tau_j \tau_{i_{s+2}} \dots \tau_{i_k} \tau_{i_k} \dots \tau_{i_{s+2}} \tau_{j+1} \tau_{i_{s-1}} \dots \tau_{i_2} \tau_{i_1} $\\
$= \tau_{i_1} \tau_{i_2} \dots \tau_{i_{s-1}} \tau_j \tau_{j+1} \tau_j \tau_{j+1} \tau_{i_{s-1}} \dots \tau_{i_2} \tau_{i_1} $.\\

(Here we assume $i_1 < j < (j+1) < i_k$. The cases $j=i_1$ and $(j+1)= i_k$ are similar and give same form of elements.)\\

So, we get some of the generators for $TW_n'$ as follows:\\
$ \{ \tau_{i_1} \tau_{i_2} \dots \tau_{i_s} ( \tau_j \tau_{j+1} \tau_j \tau_{j+1} ) \tau_{i_s} \dots \tau_{i_2} \tau_{i_1} ~  | ~ j \in \{ 1, 2, \dots n-2 \} $ and $i_1 < i_2 < \dots < i_s < j$ where $i_1, i_2, \dots ,i_s,j$ are consecutive integers $ \}. $ 

\subsubsection*{Subcase 2C: \ $i_k > j$ and $ (j+1) \notin \{ i_1, i_2, \dots , i_k \} $: \\\\ }

There is $i_s \in \{ i_1, i_2, \dots , i_k \} $ such that $i_s \le j < i_{s+1} $. \\
As $ (j+1) \notin \{ i_1, i_2, \dots , i_k \} $, $|i_{s+1} - j |>1$. So we have:\\

$S_{\tau_{i_1} \tau_{i_2} \dots \tau_{i_k}, \tau_j } \ = \ \tau_{i_1} \tau_{i_2} \dots \tau_{i_s} \tau_{i_{s+1}} \dots \tau_{i_k} \tau_j \ \overline{ ( \tau_{i_1} \tau_{i_2} \dots \tau_{i_s} \tau_{i_{s+1}} \dots \tau_{i_k} \tau_j )}^{-1} $\\
$= \ \tau_{i_1} \tau_{i_2} \dots \tau_{i_s} \tau_j \tau_{i_{s+1}} \dots \tau_{i_k} \ \overline{ ( \tau_{i_1} \tau_{i_2} \dots \tau_{i_s} \tau_j \tau_{i_{s+1}} \dots \tau_{i_k} )}^{-1}=1.$\\

So, this case does not give any nontrivial generator for $TW_n'$.
\end{proof}

\subsection{Notation:}
Let us introduce some notations as follows:\\

For $1 \le i_1 < i_2 < \dots < i_s < j \le n-2 ~ $ let us denote

 $$\alpha(i_1, i_2, \dots , i_s \ ; \ j) := \tau_{i_1} \tau_{i_2} \dots \tau_{i_s} ( \tau_j \tau_{j+1} \tau_j \tau_{j+1} ) \tau_{i_s} \dots \tau_{i_2} \tau_{i_1}, $$
 
 $$\beta(i_1, i_2, \dots , i_s \ ; \ j) := \tau_{i_1} \tau_{i_2} \dots \tau_{i_s} ( \tau_{j+1} \tau_j \tau_{j+1} \tau_j ) \tau_{i_s} \dots \tau_{i_2} \tau_{i_1}, $$
 
 $$\alpha(j) := \tau_j \tau_{j+1} \tau_j \tau_{j+1}, \ \ \ \ \ \beta(j) := \tau_{j+1} \tau_j \tau_{j+1} \tau_j.$$
 
 \medskip
  
\section{Defining Relations for $TW_n'$}\label{dr} 

 To obtain defining relations for $TW_n'$, following the Reidemeister-Schreier algorithm, we define a re-writing process $\eta$ as below. Refer \cite{mks} for more details.
$$\eta(a_{i_1}^{\epsilon_1} \dots a_{i_p}^{\epsilon_p}) := S_{K_{i_1},a_{i_1}}^{\epsilon_1} \dots S_{K_{i_p},a_{i_p}}^{\epsilon_p} \hbox{ with } \epsilon_j = 1 \hbox{ or } -1,$$
where if $\epsilon_j = 1$, $K_{i_1} = 1$ and $K_{i_j}$ = $\overline{a_{i_1}^{\epsilon_1} \dots a_{i_{j-1}}^{\epsilon_{j-1}}}, ~ j \ge 2$,  \\ and if $\epsilon_j = -1$, $K_{i_j}$ = $\overline{a_{i_1}^{\epsilon_1} \dots a_{i_j}^{\epsilon_j}}$ .

By \cite[Theorem 2.9]{mks}, the group $TW_n'$ is defined by the relations:
$$\eta(\lambda r_{\mu} \lambda^{-1})=1, ~ \lambda \in \Lambda,$$
where $r_{\mu}$ are the defining relators of $TW_n$.\\

We have the following lemma.

\begin{lemma}\label{lemma1}
	The generators $ \ \alpha(j), \ \beta(j), \ \alpha(i_1, i_2, \dots , i_s \ ; \ j), \ \beta(i_1, i_2, \dots , i_s \ ; \ j)$ satisfy the following defining relations in $TW_n'$:\\
	\begin{equation*}
	\alpha(j) \ \beta(j) = 1, \ \ \text{ for all \  }j \in \{ 1, 2, \dots , n-2 \},
	\end{equation*}
	\begin{equation*}
	\alpha(i_1, \dots , i_s \ ; \ j) \ \beta(i_1, \dots , i_s \ ; \ j) = 1, \ \ \text{ when \ } 1 \le i_1 < i_2 < \dots < i_s < j \le n-2.
	\end{equation*}
	 
\end{lemma}

\begin{proof}
	Following the Reidemeister-Schreier algorithm we will apply the re-writing process $\eta$ on all the conjugates (by the elements $\tau_{i_1} \tau_{i_2} \dots \tau_{i_k}$ of $\Lambda$) of the defining relators in $TW_n$ in order to deduce a set of defining relators for $TW_n'$.\\
	
	For all $j \in \{ 1, 2, \dots, n-1 \} $, we have the relation $\tau_j^2=1$ in $TW_n$. We apply the re-writing process $\eta$ on the conjugates of the relator as follows.\\
	 
	For any element $\tau_{i_1} \tau_{i_2} \dots \tau_{i_k} \in \Lambda$ we have,\\
	$ \eta \ (\tau_{i_1} \tau_{i_2} \dots \tau_{i_k} ( \tau_j \tau_j ) \ \tau_{i_k} \dots \tau_{i_2} \tau_{i_1}) $ \\
	$= S_{1, \tau_{i_1}} S_{ \overline{ \tau_{i_1} } , \tau_{i_2}} \dots S_{ \overline{ \tau_{i_1} \tau_{i_2} \dots \tau_{i_k}}  , \tau_j} \ S_{ \overline{ \tau_{i_1} \tau_{i_2} \dots \tau_{i_k} \tau_j}  , \tau_j} \ S_{ \overline{ \tau_{i_1} \tau_{i_2} \dots \tau_{i_k} } , \tau_{i_k} } \dots S_{ \overline{ \tau_{i_1} \tau_{i_2} } , \tau_{i_2}} S_{ \overline{ \tau_{i_1} }, \tau_{i_1} } $\\
	$= S_{ \overline{ \tau_{i_1} \tau_{i_2} \dots \tau_{i_k}}  , \tau_j} \ S_{ \overline{ \tau_{i_1} \tau_{i_2} \dots \tau_{i_k} \tau_j}  , \tau_j } .$\\
	
	For $i_k \le j$ the above expression vanishes.\\
	
	If we have $i_k > j$ and $ (j+1) \notin \{ i_1, i_2, \dots , i_k \} $ the above expression vanishes.\\
	
	In case $i_k > j$ and $ j, (j+1) \in \{ i_1, i_2, \dots , i_k \} $, assuming $j=i_s$, the above expression equals
	\begin{equation*}
		\alpha(i_1, i_2, \dots , i_{s-1} \ ; \ j) \ \beta(i_1, i_2, \dots , {i_s-1} \ ; \ j).
	\end{equation*}
	And, if $s=1$, then we have:
	\begin{equation*}
	\alpha(j) \ \beta(j).
	\end{equation*}
	
	For $i_k > j$ and $ (j+1) \in \{ i_1, i_2, \dots , i_k \} $ but $ j \notin \{ i_1, i_2, \dots , i_k \} $, assuming $j+1=i_{s+1}$, the above expression equals
	\begin{equation*}
		\beta(i_1, i_2, \dots , i_s \ ; \ j) \ \alpha(i_1, i_2, \dots , i_s \ ; \ j).
	\end{equation*}
	
	Hence, corresponding to the relation $\tau_j^2=1$ in $TW_n$ we have the following defining relations for $TW_n'$:
\begin{equation}
\alpha(j) \ \beta(j) = 1, \text{  for all } 1 \le j \le n-2,
\end{equation}
\begin{equation}
\alpha(i_1, i_2, \dots , i_s \ ; \ j) \ \beta(i_1, i_2, \dots , i_s \ ; \ j) = 1,
\end{equation}
for all $~ i_1, i_2, \dots, i_s, j ~$ such that $ ~ 1 \le i_1 < i_2 < \dots < i_s < j \le n-2.$
\end{proof}

Now, we will find the defining relations in $TW_n'$ corresponding to the defining relations $\tau_t \tau_j \tau_t \tau_j = 1, ~ |t-j|>1$, in $TW_n$.\\

We have the following lemma.
	 
\begin{lemma}\label{lemma2}
	The generators $ \ \alpha(j), \ \beta(j), \ \alpha(i_1, i_2, \dots , i_s \ ; \ j), \ \beta(i_1, i_2, \dots , i_s \ ; \ j)$ satisfy the following defining relations in $TW_n'$:\\
	
	For all $~ i_1, i_2, \dots, i_r, j, t ~$ where $ 1 \le i_1 < i_2 < \dots < i_r < t \le n-2, ~ j \le t-2,$ we have
	\begin{equation*}
	\alpha(i_1, \dots, j,~ \widehat{j+1}, \dots, i_r ; t) ~ \beta(i_1, \dots, \widehat{j},~ \widehat{j+1}, \dots, i_r ; t) = 1,
	\end{equation*}
	\begin{equation*}
	\beta(i_1, \dots, j,~ \widehat{j+1}, \dots, i_r ; t) ~ \alpha(i_1, \dots, \widehat{j},~ \widehat{j+1}, \dots, i_r ; t) = 1,
	\end{equation*}
	\begin{equation*}
	\beta(i_1, \dots, i_s,~ \widehat{j},~ j+1, \dots, i_r ; t) ~ \beta(i_1, \dots, i_s ; j) ~ \alpha(i_1, \dots, i_s,~ j,~ j+1, \dots, i_r ; t) ~ \alpha(i_1, \dots, i_s ; j) = 1,
	\end{equation*}
	\begin{equation*}
	\alpha(i_1, \dots, i_s,~ \widehat{j},~ j+1, \dots, i_r ; t) ~ \beta(i_1, \dots, i_s ; j) ~ \beta(i_1, \dots, i_s,~ j,~ j+1, \dots, i_r ; t) ~ \alpha(i_1, \dots, i_s ; j) = 1.
	\end{equation*}
	
\end{lemma}

\begin{proof}
For $ |t-j|>1 $, in $TW_n$ we have the relation: $\tau_t \tau_j \tau_t \tau_j = 1$. We rewrite this relation below.

$ \eta \ (\tau_{i_1} \tau_{i_2} \dots \tau_{i_k} ( \tau_t \tau_j \tau_t \tau_j ) \ \tau_{i_k} \dots \tau_{i_2} \tau_{i_1}) $ \\
$= S_{1, \tau_{i_1}} S_{ \overline{ \tau_{i_1} } , \tau_{i_2}} \dots S_{ \overline{ \tau_{i_1} \tau_{i_2} \dots \tau_{i_k}}  , \tau_t} S_{ \overline{ \tau_{i_1} \tau_{i_2} \dots \tau_{i_k} \tau_t}  , \tau_j} S_{\overline{ \tau_{i_1} \tau_{i_2} \dots \tau_{i_k} \tau_t \tau_j}  , \tau_t} S_{\overline{ \tau_{i_1} \tau_{i_2} \dots \tau_{i_k} \tau_t \tau_j \tau_t } , \tau_j}$\\ $S_{\overline{ \tau_{i_1} \tau_{i_2} \dots \tau_{i_k} \tau_t \tau_j \tau_t \tau_j} , \tau_{i_k} } \dots S_{\overline{ \tau_{i_1} \tau_{i_2} \dots \tau_{i_k} \tau_t \tau_j \tau_t \tau_j \tau_{i_k} \dots \tau_{i_2} } , \tau_{i_1} }$\\
$= S_{ \overline{ \tau_{i_1} \tau_{i_2} \dots \tau_{i_k}}  , \tau_t} S_{ \overline{ \tau_{i_1} \tau_{i_2} \dots \tau_{i_k} \tau_t}  , \tau_j} S_{\overline{ \tau_{i_1} \tau_{i_2} \dots \tau_{i_k} \tau_t \tau_j}  , \tau_t} S_{\overline{ \tau_{i_1} \tau_{i_2} \dots \tau_{i_k} \tau_t \tau_j \tau_t } , \tau_j}.$\\

We need to calculate the above expression in all possible cases in order to get all the remaining defining relations for $TW_n'$.\\

Without loss of generality, we may assume that $j < t$.\\

We can only have the following 3 cases:

\medskip \noindent Case 1: $i_k \le j < t$; \\
Case 2: $j < i_k \le t$;\\
Case 3: $j < t < i_k $.  

\subsection*{Case 1: $i_k \le j < t$}

In this case we have:
$$S_{ \overline{ \tau_{i_1} \tau_{i_2} \dots \tau_{i_k}}  , \tau_t}=1,$$
$$S_{ \overline{ \tau_{i_1} \tau_{i_2} \dots \tau_{i_k} \tau_t}  , \tau_j}=1,$$
$$S_{\overline{ \tau_{i_1} \tau_{i_2} \dots \tau_{i_k} \tau_t \tau_j}  , \tau_t}=1,$$
$$S_{\overline{ \tau_{i_1} \tau_{i_2} \dots \tau_{i_k} \tau_t \tau_j \tau_t } , \tau_j}=1.$$

Hence, this case gives no nontrivial defining relation for $TW_n'$.

\subsection*{Case 2: $j < i_k \le t$}

We further divide this case into 3 subcases.

\subsubsection*{Subcase 2A} $ (j+1) \in \{ i_1, i_2, \dots , i_k \} $ but $ j \notin \{ i_1, i_2, \dots , i_k \} $:  \\

Assume, $(j+1) = i_{s+1}$. Then we have:
$$S_{ \overline{ \tau_{i_1} \tau_{i_2} \dots \tau_{i_k}}  , \tau_t}=1,$$
$$S_{ \overline{ \tau_{i_1} \tau_{i_2} \dots \tau_{i_k} \tau_t}  , \tau_j}= \beta(i_1, \dots, i_s ; j),$$
$$S_{\overline{ \tau_{i_1} \tau_{i_2} \dots \tau_{i_k} \tau_t \tau_j}  , \tau_t}=1,$$
$$S_{\overline{ \tau_{i_1} \tau_{i_2} \dots \tau_{i_k} \tau_t \tau_j \tau_t } , \tau_j}=\alpha(i_1, \dots, i_s ; j).$$

Hence, we get the relations:
$$\beta(i_1, \dots, i_s ; j) ~ \alpha(i_1, \dots, i_s ; j) = 1.$$

\subsubsection*{Subcase 2B} $ j, (j+1) \in \{ i_1, i_2, \dots , i_k \} $: \\ 

Assume, $(j+1) = i_{s+1},~ j = i_s$. Then we have:

$$S_{ \overline{ \tau_{i_1} \tau_{i_2} \dots \tau_{i_k}}  , \tau_t}=1,$$
$$S_{ \overline{ \tau_{i_1} \tau_{i_2} \dots \tau_{i_k} \tau_t}  , \tau_j}=\alpha(i_1, \dots, i_{s-1} ; j),$$
$$S_{\overline{ \tau_{i_1} \tau_{i_2} \dots \tau_{i_k} \tau_t \tau_j}  , \tau_t}=1,$$
$$S_{\overline{ \tau_{i_1} \tau_{i_2} \dots \tau_{i_k} \tau_t \tau_j \tau_t } , \tau_j}=\beta(i_1, \dots, i_{s-1} ; j).$$

So, we get the relations:
$$\alpha(i_1, \dots, i_{s-1} ; j) ~ \beta(i_1, \dots, i_{s-1} ; j) = 1.$$

\subsubsection*{Subcase 2C} $ (j+1) \notin \{ i_1, i_2, \dots , i_k \} $:  In this case we have:
$$S_{ \overline{ \tau_{i_1} \tau_{i_2} \dots \tau_{i_k}}  , \tau_t}=1,$$
$$S_{ \overline{ \tau_{i_1} \tau_{i_2} \dots \tau_{i_k} \tau_t}  , \tau_j}=1,$$
$$S_{\overline{ \tau_{i_1} \tau_{i_2} \dots \tau_{i_k} \tau_t \tau_j}  , \tau_t}=1,$$
$$S_{\overline{ \tau_{i_1} \tau_{i_2} \dots \tau_{i_k} \tau_t \tau_j \tau_t } , \tau_j}=1.$$
So, we do not get any nontrivial relation from this subcase.

\subsection*{Case 3: $j < t < i_k $} We need to divide this case into 9 subcases.

\subsubsection*{Subcase 3A}
$ (j+1) \in \{ i_1, i_2, \dots , i_k \} $, $ j \notin \{ i_1, i_2, \dots , i_k \} $, $ (t+1) \in \{ i_1, i_2, \dots , i_k \} $, $ t \notin \{ i_1, i_2, \dots , i_k \} $: \\

Assume, $(j+1)=i_{s+1}, ~ (t+1)=i_{r+1}$. In this case we have:
$$S_{ \overline{ \tau_{i_1} \tau_{i_2} \dots \tau_{i_k}}  , \tau_t}=\beta(i_1, \dots, \widehat{j}, \dots, i_r ; t),$$
$$S_{ \overline{ \tau_{i_1} \tau_{i_2} \dots \tau_{i_k} \tau_t}  , \tau_j}=\beta(i_1, \dots, i_s ; j),$$
$$S_{\overline{ \tau_{i_1} \tau_{i_2} \dots \tau_{i_k} \tau_t \tau_j}  , \tau_t}=\alpha(i_1, \dots, j, \dots, i_r ; t),$$
$$S_{\overline{ \tau_{i_1} \tau_{i_2} \dots \tau_{i_k} \tau_t \tau_j \tau_t } , \tau_j}=\alpha(i_1, \dots, i_s ; j).$$

($~ \widehat{j} ~$ denotes absence of $j$)\\

Hence, we get the relations:
$$\beta(i_1, \dots, i_s,~ \widehat{j},~ j+1, \dots, i_r ; t) ~ \beta(i_1, \dots, i_s ; j) ~ \alpha(i_1, \dots, i_s,~ j,~ j+1, \dots, i_r ; t) ~ \alpha(i_1, \dots, i_s ; j) = 1.$$

\subsubsection*{Subcase 3B}
$ (j+1) \in \{ i_1, i_2, \dots , i_k \} $, $ j \notin \{ i_1, i_2, \dots , i_k \} $, and $~ t, (t+1) \in \{ i_1, i_2, \dots , i_k \}: $\\

Assume, $(j+1)=i_{s+1}, ~ (t+1)=i_{r+1}, ~ t=i_r$. In this case we have:
$$S_{ \overline{ \tau_{i_1} \tau_{i_2} \dots \tau_{i_k}}  , \tau_t}=\alpha(i_1, \dots, \widehat{j}, \dots, i_{r-1} ; t),$$
$$S_{ \overline{ \tau_{i_1} \tau_{i_2} \dots \tau_{i_k} \tau_t}  , \tau_j}=\beta(i_1, \dots, i_s ; j),$$
$$S_{\overline{ \tau_{i_1} \tau_{i_2} \dots \tau_{i_k} \tau_t \tau_j}  , \tau_t}=\beta(i_1, \dots, j, \dots, i_{r-1} ; t),$$
$$S_{\overline{ \tau_{i_1} \tau_{i_2} \dots \tau_{i_k} \tau_t \tau_j \tau_t } , \tau_j}=\alpha(i_1, \dots, i_s ; j).$$

So, we get the relations:
$$\alpha(i_1, \dots, i_s,~ \widehat{j},~ j+1, \dots, i_{r-1} ; t) ~ \beta(i_1, \dots, i_s ; j) ~ \beta(i_1, \dots, i_s,~ j,~ j+1, \dots, i_{r-1} ; t) ~ \alpha(i_1, \dots, i_s ; j) = 1.$$

\subsubsection*{Subcase 3C}
$ (j+1) \in \{ i_1, i_2, \dots , i_k \} $, $ j \notin \{ i_1, i_2, \dots , i_k \} $, and $~ (t+1) \notin \{ i_1, i_2, \dots , i_k \}: $\\

Assume, $(j+1)=i_{s+1}$. 
In this case we have:
$$S_{ \overline{ \tau_{i_1} \tau_{i_2} \dots \tau_{i_k}}  , \tau_t}=1,$$
$$S_{ \overline{ \tau_{i_1} \tau_{i_2} \dots \tau_{i_k} \tau_t}  , \tau_j}=\beta(i_1, \dots, i_s ; j),$$
$$S_{\overline{ \tau_{i_1} \tau_{i_2} \dots \tau_{i_k} \tau_t \tau_j}  , \tau_t}=1,$$
$$S_{\overline{ \tau_{i_1} \tau_{i_2} \dots \tau_{i_k} \tau_t \tau_j \tau_t } , \tau_j}=\alpha(i_1, \dots, i_s ; j).$$

Hence, we get the relations:
$$\beta(i_1, \dots, i_s ; j) ~ \alpha(i_1, \dots, i_s ; j) = 1.$$

\subsubsection*{Subcase 3D}
$j, (j+1) \in \{ i_1, i_2, \dots , i_k \}$, $ (t+1) \in \{ i_1, i_2, \dots , i_k \} $, $ t \notin \{ i_1, i_2, \dots , i_k \}: $\\

Assume, $(j+1)=i_{s+1}, ~ j=i_s, ~ (t+1)=i_{r+1}$. In this case we have:
$$S_{ \overline{ \tau_{i_1} \tau_{i_2} \dots \tau_{i_k}}  , \tau_t}=\beta(i_1, \dots, j, \dots, i_r ; t),$$
$$S_{ \overline{ \tau_{i_1} \tau_{i_2} \dots \tau_{i_k} \tau_t}  , \tau_j}=\alpha(i_1, \dots, i_{s-1} ; j),$$
$$S_{\overline{ \tau_{i_1} \tau_{i_2} \dots \tau_{i_k} \tau_t \tau_j}  , \tau_t}=\alpha(i_1, \dots, \widehat{j}, \dots, i_r ; t),$$
$$S_{\overline{ \tau_{i_1} \tau_{i_2} \dots \tau_{i_k} \tau_t \tau_j \tau_t } , \tau_j}=\beta(i_1, \dots, i_{s-1} ; j).$$

So, we get the relations:

\medskip \noindent 
$\beta(i_1, \dots, i_{s-1},~ j,~ j+1, \dots, i_r ; t) ~ \alpha(i_1, \dots, i_{s-1} ; j) ~ \alpha(i_1, \dots, i_{s-1},~ \widehat{j},~ j+1, \dots, i_r ; t) ~ \\ \beta(i_1, \dots, i_{s-1} ; j) = 1$.

\subsubsection*{Subcase 3E}
$j, (j+1) \in \{ i_1, i_2, \dots , i_k \}$, and $~ t, (t+1) \in \{ i_1, i_2, \dots , i_k \} $:\\

Assume, $(j+1)=i_{s+1}, ~ j=i_s, ~ (t+1)=i_{r+1}, ~ t=i_r$. In this case we have:
$$S_{ \overline{ \tau_{i_1} \tau_{i_2} \dots \tau_{i_k}}  , \tau_t}=\alpha(i_1, \dots, j, \dots, i_{r-1} ; t),$$
$$S_{ \overline{ \tau_{i_1} \tau_{i_2} \dots \tau_{i_k} \tau_t}  , \tau_j}=\alpha(i_1, \dots, i_{s-1} ; j),$$
$$S_{\overline{ \tau_{i_1} \tau_{i_2} \dots \tau_{i_k} \tau_t \tau_j}  , \tau_t}=\beta(i_1, \dots, \widehat{j}, \dots, i_{r-1} ; t),$$
$$S_{\overline{ \tau_{i_1} \tau_{i_2} \dots \tau_{i_k} \tau_t \tau_j \tau_t } , \tau_j}=\beta(i_1, \dots, i_{s-1} ; j).$$

Hence, we get the relations:

\medskip \noindent   $\alpha(i_1, \dots, i_{s-1},~ j, ~j+1, \dots, i_{r-1} ; t)  ~\alpha(i_1, \dots, i_{s-1} ; j)  ~\beta(i_1, \dots, i_{s-1},~\widehat{j},~j+1, \dots, ~i_{r-1} ; t)~ \\ 
\beta(i_1, \dots, i_{s-1} ; j) = 1$.

\subsubsection*{Subcase 3F}
$j, (j+1) \in \{ i_1, i_2, \dots , i_k \}$, and $~ (t+1) \notin \{ i_1, i_2, \dots , i_k \} $:\\

Assume, $(j+1)=i_{s+1}, ~ j=i_s$.  In this case we have:
$$S_{ \overline{ \tau_{i_1} \tau_{i_2} \dots \tau_{i_k}}  , \tau_t}=1,$$
$$S_{ \overline{ \tau_{i_1} \tau_{i_2} \dots \tau_{i_k} \tau_t}  , \tau_j}=\alpha(i_1, \dots, i_{s-1} ; j),$$
$$S_{\overline{ \tau_{i_1} \tau_{i_2} \dots \tau_{i_k} \tau_t \tau_j}  , \tau_t}=1,$$
$$S_{\overline{ \tau_{i_1} \tau_{i_2} \dots \tau_{i_k} \tau_t \tau_j \tau_t } , \tau_j}=\beta(i_1, \dots, i_{s-1} ; j).$$

So, we get the relations:
$$\alpha(i_1, \dots, i_{s-1} ; j) ~ \beta(i_1, \dots, i_{s-1} ; j) = 1.$$

\subsubsection*{Subcase 3G}
$(j+1) \notin \{ i_1, i_2, \dots , i_k \} $, and $ (t+1) \in \{ i_1, i_2, \dots , i_k \} $, $ t \notin \{ i_1, i_2, \dots , i_k \} $:\\

Assume, $(t+1)=i_{r+1}$.\\

In this case we have:\\

$S_{ \overline{ \tau_{i_1} \tau_{i_2} \dots \tau_{i_k}}  , \tau_t} = \begin{cases}
\beta(i_1, \dots, j, \dots, i_r ; t) & \text{ if $j \in \{ i_1, i_2, \dots , i_k \},$ } \\
\beta(i_1, \dots, \widehat{j}, \dots, i_r ; t) & \text{ if $j \notin \{ i_1, i_2, \dots , i_k \},$ }
\end{cases}$ \\

$S_{ \overline{ \tau_{i_1} \tau_{i_2} \dots \tau_{i_k} \tau_t}  , \tau_j}$ $=1$,

$S_{\overline{ \tau_{i_1} \tau_{i_2} \dots \tau_{i_k} \tau_t \tau_j}  , \tau_t} = \begin{cases}
\alpha(i_1, \dots, \widehat{j}, \dots, i_r ; t) & \text{ if $j \in \{ i_1, i_2, \dots , i_k \},$ } \\
\alpha(i_1, \dots, j, \dots, i_r ; t) & \text{ if $j \notin \{ i_1, i_2, \dots , i_k \},$ }
\end{cases}$ \\

$S_{\overline{ \tau_{i_1} \tau_{i_2} \dots \tau_{i_k} \tau_t \tau_j \tau_t } , \tau_j}$ $=1$.\\

So, we get the relations:
$$\beta(i_1, \dots, j, \dots, i_r ; t) ~ \alpha(i_1, \dots, \widehat{j}, \dots, i_r ; t) = 1,$$
$$\beta(i_1, \dots, \widehat{j}, \dots, i_r ; t) ~ \alpha(i_1, \dots, j, \dots, i_r ; t) = 1.$$

\subsubsection*{Subcase 3H}
$(j+1) \notin \{ i_1, i_2, \dots , i_k \} $, and $~ t, (t+1) \in \{ i_1, i_2, \dots , i_k \} $:\\

Assume, $(t+1)=i_{r+1}, ~ t=i_r$. In this case we have:\\

$S_{ \overline{ \tau_{i_1} \tau_{i_2} \dots \tau_{i_k}}  , \tau_t} = \begin{cases}
\alpha(i_1, \dots, j, \dots, i_{r-1} ; t) & \text{ if $j \in \{ i_1, i_2, \dots , i_k \},$ } \\
\alpha(i_1, \dots, \widehat{j}, \dots, i_{r-1} ; t) & \text{ if $j \notin \{ i_1, i_2, \dots , i_k \},$ }
\end{cases}$ \\

$S_{ \overline{ \tau_{i_1} \tau_{i_2} \dots \tau_{i_k} \tau_t}  , \tau_j}$ $=1$,

$S_{\overline{ \tau_{i_1} \tau_{i_2} \dots \tau_{i_k} \tau_t \tau_j}  , \tau_t} = \begin{cases}
\beta(i_1, \dots, \widehat{j}, \dots, i_{r-1} ; t) & \text{ if $j \in \{ i_1, i_2, \dots , i_k \},$ } \\
\beta(i_1, \dots, j, \dots, i_{r-1} ; t) & \text{ if $j \notin \{ i_1, i_2, \dots , i_k \},$ }
\end{cases}$ \\

$S_{\overline{ \tau_{i_1} \tau_{i_2} \dots \tau_{i_k} \tau_t \tau_j \tau_t } , \tau_j}$ $=1$.\\

Hence, we get the relations:
$$\alpha(i_1, \dots, j, \dots, i_{r-1} ; t) ~ \beta(i_1, \dots, \widehat{j}, \dots, i_{r-1} ; t) = 1,$$
$$\alpha(i_1, \dots, \widehat{j}, \dots, i_{r-1} ; t) ~ \beta(i_1, \dots, j, \dots, i_{r-1} ; t) = 1.$$

\subsubsection*{Subcase 3I}
$(j+1) \notin \{ i_1, i_2, \dots , i_k \} $, and $(t+1) \notin \{ i_1, i_2, \dots , i_k \} $:\\

In this case we have:
$$S_{ \overline{ \tau_{i_1} \tau_{i_2} \dots \tau_{i_k}}  , \tau_t}=1,$$
$$S_{ \overline{ \tau_{i_1} \tau_{i_2} \dots \tau_{i_k} \tau_t}  , \tau_j}=1,$$
$$S_{\overline{ \tau_{i_1} \tau_{i_2} \dots \tau_{i_k} \tau_t \tau_j}  , \tau_t}=1,$$
$$S_{\overline{ \tau_{i_1} \tau_{i_2} \dots \tau_{i_k} \tau_t \tau_j \tau_t } , \tau_j}=1.$$

So, we do not get any nontrivial relation from this subcase.\\

Collecting the relations obtained in all the above cases we have the lemma.
\end{proof}

\section{Finite presentation for $TW_n$: Proof of the theorems}\label{simp}
In this section, we will simplify the presentation for $TW_n'$ that we deduced in the previous section. We will apply Tietze transformations on the current presentation for $TW_n'$ in order to deduce an equivalent presentation for $TW_n'$ with less number of generators and relations than the last one. Refer to \cite{mks} for more details on Tietze transformations. We begin with the following lemma.

\begin{lemma}\label{lemma3}
	For $n \ge 3$, $TW_n'$ has the following presentation:\\
	
	Generators: $~ ~ ~ \beta(j),~ \beta(i_1, i_2, \dots , i_s \ ; \ j)$, for $~ 1 \le i_1 < i_2 < \dots < i_s < j \le n-2$, 
	
\medskip 	Defining relations:
	$$\beta(i_1, \dots, i_s,~ j,~ \widehat{j+1}, \dots, i_r ; t) = \beta(i_1, \dots, i_s,~ \widehat{j},~ \widehat{j+1}, \dots, i_r ; t), $$
	$$\beta(i_1, \dots, i_s,~ j,~ j+1, \dots, i_r ; t) = \beta(i_1, \dots, i_s ; j)^{-1} ~ \beta(i_1, \dots, i_s,~ \widehat{j},~ j+1, \dots, i_r ; t) ~ \beta(i_1, \dots, i_s ; j), $$
	where $ 1 \le i_1 < i_2 < \dots < i_s < j < \dots < i_r < t \le n-2, ~ j \le t-2$.
\end{lemma}

\begin{proof}
	From \lemref{lemma1}, we have $\alpha(j)=\beta(j)^{-1}$, $\alpha(i_1, i_2, \dots , i_s \ ; \ j)=\beta(i_1, i_2, \dots , i_s \ ; \ j)^{-1}$.\\
	Hence, we replace $\alpha(j)$ by $\beta(j)^{-1}$ and $\alpha(i_1, i_2, \dots , i_s \ ; \ j)$ by $\beta(i_1, i_2, \dots , i_s \ ; \ j)^{-1}$ in all other defining relations for $TW_n'$, and remove all $\alpha(j),~ \alpha(i_1, i_2, \dots , i_s \ ; \ j)$ from the set of generators. This completes the proof of \lemref{lemma3}.
\end{proof}

\subsection{Observation}

Note that, we have the defining relations
$$\beta(i_1, \dots, i_s,~ j,~ \widehat{j+1}, \dots, i_r ; t) = \beta(i_1, \dots, i_s,~ \widehat{j},~ \widehat{j+1}, \dots, i_r ; t).$$ 

Note that here we have $j \le t-2$. Let us look at the following example.\\

Consider the generator $\beta(3,4,6,7,9,10,11;12)$ in $TW_{15}'$. From the above set of relations, as `5' does not appear in $\beta(3,4,6,7,9,10,11;12)$, we can conclude that $\beta(3,4,6,7,9,10,11;12) = \beta(3,6,7,9,10,11;12)$.
As `4' is missing in $\beta(3,6,7,9,10,11;12)$, we get $\beta(3,6,7,9,10,11;12) = \beta(6,7,9,10;11)$.
We can go further. Using the same relations we get $\beta(6,7,9,10;11) = \beta(6,9,10;11) = \beta(9,10;11)$.\\

From the above observation it is clear that using the above defining relations finitely many times, any generator $\beta(i_1, i_2, \dots, i_s; j)$ can be shown to be equal to a generator of the form $\beta(j-p, j-p+1, \dots, j-1 \ ; \ j)$ for some $p < j$, or be equal to $\beta(j)$. Let's call these the \textit{normal forms} of the generators.\\

\subsection{Notation: }
We will follow the notations for the normal forms as below:
$$ \text{For } 1 \le p < j, \ \ \beta_{p}(j) := \beta(j-p, \dots, j-1 ; j), \ \ \text{ and }\ \ \ \beta_{0}(j) := \beta(j). $$

As every generator is equal to its normal form, we replace all the generators with their normal forms in all the defining relations and remove all the generators except the normal forms from the generating set. For clarity of exposition, we define the following.

\begin{definition}
	For a generator $~ \beta(i_1, i_2, \dots, i_s ; j)~$ we define the \it{ highest missing entry in $\beta(i_1, i_2, \dots, i_s ; j)$ } to be $k$ for some $~i_1-1 \le k \le j-1~$ if $k \notin \{ i_1, i_2, \dots, i_s, j \} $ but for any $m$ with $k < m \le j$, $m \in \{ i_1, i_2, \dots, i_s, j \}$.
\end{definition}
\subsection{Proof of \thmref{mainth}}

\begin{proof}
	 We have the following relations in the presentation for $TW_n'$, $n \geq 3$,  as in \lemref{lemma3}:
	 $$\beta(i_1, \dots, i_s,~ j,~ j+1, \dots, i_r ; t) = \beta(i_1, \dots, i_s ; j)^{-1} ~ \beta(i_1, \dots, i_s,~ \widehat{j},~ j+1, \dots, i_r ; t) ~ \beta(i_1, \dots, i_s ; j),$$
	 where $ 1 \le i_1 < i_2 < \dots < i_s < j < \dots < i_r < t \le n-2, ~ j \le t-2$.\\
	 
	 We replace the generators appearing in these relations by their normal forms $\beta_{p}(j)$'s. Our goal is to find the modified relations after the substitution.

 Consider the left hand side of the above relations. We have $\beta(i_1, \dots, i_s,~ j,~ j+1, \dots, i_r ; t)$. Note that the highest missing entry in $\beta(i_1, \dots, i_s,~ j,~ j+1, \dots, i_r ; t)$ cannot be $j$ or $j+1$, as both are present as entries. So we can have 2 possibilities. We examine the 2 cases separately below.\\
 
 $\textbf{Case 1:}$ The highest missing entry in $\beta(i_1, \dots, i_s,~ j,~ j+1, \dots, i_r ; t)$ is greater than $j+1$.\\
 
Suppose the highest missing entry in $\beta(i_1, \dots, i_s,~ j,~ j+1, \dots, i_r ; t)$ is $j+(l-1)$ for some $l \ge 3$. Also suppose the highest missing entry in $\beta(i_1, \dots, i_s; j)$ is $m-1$ for some $1 \le m \le j$.\\
 
Then, the relations are equivalent to the following relations:\\
$\text{ for all } l \ge 3,~ 1 \le m \le j,~ j \le t-2,$
 \begin{equation*}
	\beta(j+l, \dots, t-1;t) = \beta(m, \dots, j-1;j)^{-1} \beta(j+l, \dots, t-1;t) \beta(m, \dots, j-1;j).
 \end{equation*}
 
 So, after the substitution by normal forms the relations become:
 \begin{equation*}
	 \beta_{t-(j+l)}(t) = \beta_{j-m}(j)^{-1} ~ \beta_{t-(j+l)}(t) ~ \beta_{j-m}(j), ~ ~ \text{ for all } l \ge 3,~ 1 \le m \le j,~ j \le t-2.
 \end{equation*}
 
Equivalently,
\begin{equation*}
\beta_{j-m}(j) ~ \beta_{t-(j+l)}(t) =  \beta_{t-(j+l)}(t) ~ \beta_{j-m}(j), ~ ~ \text{ for all } l \ge 3,~ 1 \le m \le j,~ j \le t-2.
\end{equation*}

$\textbf{Case 2:}$ The highest missing entry in $\beta(i_1, \dots, i_s,~ j,~ j+1, \dots, i_r ; t)$ is less than $j$.\\

Suppose the highest missing entry in $\beta(i_1, \dots, i_s,~ j,~ j+1, \dots, i_r ; t)$ is $m-1$ for some $1 \le m \le j$. Then clearly the highest missing entry in $\beta(i_1, \dots, i_s; j)$ is also $m-1$.\\

So, after the substitution by normal forms the relations become:
\begin{equation*}
\beta_{t-m}(t) = \beta_{j-m}(j)^{-1} ~ \beta_{t-(j+1)}(t) ~ \beta_{j-m}(j), ~ ~ \text{ for all } 1 \le m \le j,~ j \le t-2.
\end{equation*}
	 
	This proves the theorem. 
\end{proof}

\subsection{Further elimination:} We shall further reduce the number of generators in the presentation by removing all  $\beta_p(j)$ with $p > 1$ by using the defining relations:
\begin{equation*}
\beta_{t-m}(t) = \beta_{j-m}(j)^{-1} ~ \beta_{t-(j+1)}(t) ~ \beta_{j-m}(j),
\end{equation*}
for all $~ m,~ j,~ t \in \{ 1,\dots,n-2 \}$ with $~ 1 \le m \le j \le t-2.$\\

Note that, if we consider the cases where $j=m$ in the above set of relations, we obtain the following set of relations:
\begin{equation*}
\beta_{t-m}(t) = \beta_{0}(m)^{-1} ~ \beta_{t-(m+1)}(t) ~ \beta_{0}(m),
\end{equation*}
for all $~ m,~ t \in \{ 1,\dots,n-2 \} ~$ with $~ 1 \le m \le t-2.$\\

So, if $~ t-m \ge 2, ~$ we can express $~ \beta_{t-m}(t) ~$ as the conjugate of $~ \beta_{t-(m+1)}(t) ~$ by $~ \beta_{0}(m)$. We do this iteratively to express $~ \beta_{t-m}(t) ~$ as the conjugate of $\beta_1(t)$ by the element $\beta_0(t-2) \dots \beta_0(m)$ and thus remove all $\beta_p(j)$ with $p \ge 2$ from the set of generators after replacing them with the above values in all the remaining relations.

\begin{lemma}\label{lemma5}
	For $n \ge 3$, $TW_n'$ has a finite presentation with $(2n-5) $ generators.
\end{lemma}

\begin{proof}
	After performing the above substitution we are left with $\beta_p(j)$ with $p \le 1$ and $1 \le j \le n-2.$ Hence, corresponding to every $~ 2 \le j \le n-2 ~$ we have 2 generators $\beta_0(j)$ and $\beta_1(j)$. For $j=1$, we have only 1 generator, namely $\beta_0(1)$. So, we have total $~ 2 \times (n-3) + 1 = 2n-5 ~$ generators in the final presentation for $TW_n'$ for $n \ge 3.$
	
	Note that the presentation given in \thmref{mainth} has finitely many defining relations. As finitely many  $\beta_p(j)$ are being replaced and each $\beta_p(j)$ appears finitely many times in all the defining relations, after the above substitution we will have finitely many defining relations in the final presentation. This proves the lemma.
\end{proof}

\subsection*{Proof of \thmref{thmrank}}

	We consider the abelianization of $TW_n'$ for $n \ge 3,$ $~(TW_n')^{ab} = TW_n'/TW_n''$. In order to find a presentation for $(TW_n')^{ab}$ we insert all possible commuting relations $\beta_p(j)~ \beta_q(i) = \beta_q(i)~ \beta_p(j),~$ for all $i,j \in \{1, \dots, n-2\},~ 0 \le p < j,~ 0 \le q < i, ~$ in the presentation for $TW_n'$. This gives the following presentation for $(TW_n')^{ab}$:\\
	
	Generators: $ \ \ \ \ \beta_{p}(j), \ \ \ \ \  0 \le p < j \le n-2. $\\
	
	Defining relations: $\beta_p(j)~ \beta_q(i) = \beta_q(i)~ \beta_p(j),~~ \forall i,j \in \{1, \dots, n-2\}, $
	\begin{equation*}
	\beta_{t-m}(t) =  \beta_{t-(j+1)}(t),~~ 1 \le m \le j,~~ j+2 \le t \le n-2.
	\end{equation*}
	
	Iterating the last set of relations, we deduce that $\beta_p(j) = \beta_1(j)$ for all $p \ge 2$ and for all $j \ge 3$. Hence, we remove all $\beta_p(j)$ with $p \ge 2$ from the set of generators by replacing them with $\beta_1(j)$. After this replacement we get the following presentation for $(TW_n')^{ab}$:\\
	
	Generators: $~ \beta_0(1),~ \beta_0(j),~ \beta_1(j),~ 2 \le j \le n-2.$\\
	
	Defining relations: $\beta_p(j)~ \beta_q(i) = \beta_q(i)~ \beta_p(j),~~ \forall i,j \in \{1, \dots, n-2\},~~ p, q \in \{0,1\}.$\\
	
	Clearly, this is the presentation for direct sum of $(2n-5)$ copies of $\Z,~~$ i.e. $\Z^{2n-5}$.
	So, $~ (TW_n')^{ab} ~$ is isomorphic to $~ \Z^{2n-5}$. Hence, rank of $~ (TW_n')^{ab} ~$ is $~ (2n-5)$.\\
	
	As, $~ (TW_n')^{ab} ~$ is the homomorphic image of $~TW_n'~$ under the quotient homomorphism $~TW_n' \longrightarrow (TW_n')^{ab},~$ rank of $~ (TW_n')^{ab} ~$ is less than or equal to the rank of $~TW_n'.~$ Thus, rank($~TW_n'~$) $\ge$ rank($~ (TW_n')^{ab} ~$) $= 2n-5.$ From \lemref{lemma5} we get rank($~TW_n'~$) $\le 2n-5$. So, we conclude that rank($~TW_n'~$) $= 2n-5$.

\subsection*{Proof of \corref{cor1}:}
In the proof of \thmref{thmrank} we observed that for $n \ge 3$, $TW_n'/TW_n''$ is isomorphic to direct sum of $(2n-5)$ copies of $\Z$. So, we conclude that $TW_n' \ne TW_n''$, hence $TW_n'$ is not perfect for any $n \ge 3.$\\

For $n \le 5$, $TW_n'$ are well known groups. We have the following.

\begin{prop}\label{prop7}
	We have the following:\\
	(i) $TW_2'$ is the identity group $\{ 1 \}$.\\
	(ii) $TW_3'$ is the infinite cyclic group $\Z$.\\
	(iii) $TW_4'$ and $TW_5'$ are free groups of rank $3$ and $5$, respectively. 
\end{prop}

\begin{proof}
	Note that, $TW_2 ~ = ~ \langle ~ \tau_1 ~|~ \tau_1^2=1 ~ \rangle ~ = ~ \Z/2\Z ~$ and $~ \Z/2\Z ~$ is an abelian group.
	Hence, $TW_2'$ is the identity group.
	
	From \thmref{mainth} it follows that $~ TW_3' ~ = ~ \langle ~ \beta_0(1) ~ \rangle ~$ which is isomorphic to the infinite cyclic group $\Z$.
	
	From \thmref{mainth} it follows that $~ TW_4' ~ = ~ \langle ~ \beta_0(1),~ \beta_0(2),~ \beta_1(2) ~ \rangle ~$ which is the free group of rank 3.
	
	From \thmref{mainth} it follows that: $$~ TW_5' ~ = ~ \langle ~ \beta_0(1),~ \beta_0(2),~ \beta_1(2),~ \beta_0(3),~ \beta_1(3),~ \beta_2(3) ~ | ~ \beta_2(3) ~ = ~ \beta_0(1)^{-1} ~ \beta_1(3) ~ \beta_0(1) ~ \rangle ~$$
	$$= ~ \langle ~ \beta_0(1),~ \beta_0(2),~ \beta_1(2),~ \beta_0(3),~ \beta_1(3)~ \rangle .~  \ \ \ \ \ \ \ \ \ \ \ \ \ \ \ \ \ \ \ \ \ \ \ \ \ \ \ \ \ \ \ \ \ \ \ \ \ \ \ \ \ $$
	Hence, $TW_5'$ is free of rank 5.
    This completes the proof of \propref{prop7}.
\end{proof}

From \cite{pv} we have a necessary and sufficient condition for the commutator subgroup of a right-angled Coxeter group to be free. Since $TW_n$ is a right-angled Coxeter group, we check the condition for $TW_n$. We note the following definitions.

\begin{definition}
	A graph $\Gamma$ is called \textit{ chordal } if for every cycle in $\Gamma$ with atleast 4 vertices there is an edge (called chord) in $\Gamma$ joining 2 non-adjacent vertices of the cycle.
\end{definition}

\begin{definition}
	The \textit{Coxeter graph} $~ \Gamma_{TW_n} ~$ corresponding to $~ TW_n ~$ is defined as follows. Corresponding to each generator $\tau_i$ of $TW_n$ we have a vertex $v_i$ in $~ \Gamma_{TW_n}. ~$ Corresponding to each commuting defining relation $\tau_i \tau_j = \tau_j \tau_i,~ |i-j|>1,~$ we have an edge in $~ \Gamma_{TW_n} ~$ joining $v_i$ and $v_j$.
\end{definition}

We have the following proposition.

\begin{prop}\label{prop8}
	For $n \ge 6$, $TW_n'$ is not a free group.
\end{prop}

\begin{proof}

As proved in \cite{pv}, for a right-angled Coxeter group $G$, the commutator subgroup $G'$ is free group if and only if the Coxeter graph of $G$, $~ \Gamma_{G} ~$ is chordal.

\begin{figure}[ht!]
	\centering
	\includegraphics[width=35mm]{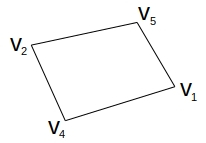}
	\caption{Cycle with 4 vertices but no chord in $~ \Gamma_{TW_n}, ~$ $n \ge 6$ \label{chordal}}
\end{figure}

Consider the Coxeter graph $~ \Gamma_{TW_n} $ corresponding to $~ TW_n. ~$ Note that for $n \ge 6$, $~ \Gamma_{TW_n} ~$ contains the cycle $ v_1 v_4 v_2 v_5 v_1 $ joining the vertices $~ v_1, v_4, v_2, v_5 ~$  (as in the figure above). Clearly this cycle does not have any chord; as $\tau_1, \tau_2$ do not commute and $\tau_4, \tau_5$ do not commute. This shows that for $~ n \ge 6 ~$ $\Gamma_{TW_n} ~$ is not chordal.

Consequently, $~ TW_n' ~$ is not free for $ n \ge 6$, proving \propref{prop8}.
\end{proof}

\subsection*{Proof of \corref{corfree}} 

\corref{corfree} follows from \propref{prop7}.  
and \propref{prop8}.

\subsection*{Presentation for $TW_6'$} As follows from the above, $TW_6'$ is the first non-free group in the family of $TW_n'$, $n \geq 3$.  Here, we note down a presentation for $TW_6'$ with minimal number of generators:\\

Generators: $\beta_0(1),~ \beta_0(2),~ \beta_1(2),~ \beta_0(3),~ \beta_1(3),~ \beta_0(4),~ \beta_1(4).$\\

Defining relations:
$$ \beta_0(1) ~ \beta_0(4) ~ = ~ \beta_0(4) ~ \beta_0(1), $$
$$ \beta_1(2)^{-1} ~ \beta_1(4) ~ \beta_1(2) ~ = ~ \beta_0(1)^{-1} ~ \beta_0(2)^{-1} ~ \beta_1(4) ~ \beta_0(2) ~ \beta_0(1). $$

\bigskip

\bigskip

\begin{ack} Thanks to Andrei Vesnin,  Matt Zaremsky and Pranab Sardar for comments on this work. 

The work was initiated when Soumya Dey was visiting the Sobolev Institute of Mathematics, Novosibirsk, Russia during July 2017, and he is indebted to Mahender Singh for facilitating the visit by DST grant INT/RUS/RSF/P-2. Dey is grateful to Andrei Vesnin for introducing him to the twin groups and suggesting the problem.  Dey acknowledges initial discussions with Valeriy Bardakov on $TW_3'$ and $TW_4'$. This problem was a part of the Indo-Russian collaboration, supported by the above DST grant and the grant RSF-16-41-02006.

  This research was also supported in part by the International Centre for Theoretical Sciences (ICTS) during a visit for participating in the program - Geometry, Groups and Dynamics (Code: ICTS/ggd2017/11).   Thanks to ICTS for the hospitality during the work. 

\end{ack}


\end{document}